\documentclass[11pt]{article}
\usepackage[usenames]{color}
\usepackage{amsmath,amsfonts,amsthm}
\usepackage{setspace} 
\usepackage{hyperref} 
\usepackage{syntonly} 
\usepackage[mathcal]{euscript} 
\usepackage[all]{xypic}
\usepackage{color}

\def\NN{{\mathbb N}}

\newcommand{\FDer}[1]{\stackrel{#1}{\to}}

\newcommand{\Comp}[1]{\cC (#1)}

\newcommand{\ctp}{\widehat{\otimes}}
\newcommand{\Tot}{\hbox{\rm{Tot}}}
\newcommand{\Cech}{ \check{\rm{C}}}
\newcommand{\SigmaFinite}[1]{$\Sigma$-\hbox{finite#1}}

\def\cM{{\cal M}}

\def\cC{{\cal C}}

\def\cV{{\cal V}}

\def\cK{{\cal K}}

\def\cV{{\cal V}}

\def\Tor{\hbox{{\rm{Tor}}}}

\def\Ker{\hbox{{\rm{Ker}}}}

\def\gr{\hbox{{\rm{gr}}}}

\def\Ext{{\hbox{\rm{Ext}}}}
\def\Ass{{\hbox{\rm {Ass}}}}
\def\Hom{{\hbox{\rm{Hom}}}}

\def\Im{{\hbox{\rm{Im}}}}
\def\Depth{{\hbox{\rm{depth}}}}

\def\Length{{\hbox{\rm{length}}}}

\def\Ord{{\hbox{\rm{ord}}}}
\def\cd{{\hbox{\rm{cd}}}}
\def\Max{{\hbox{\rm{Max}}}}

\newtheorem{Teo}{Theorem}[section]

\newtheorem{Lemma}[Teo]{Lemma}
\newtheorem{Prop}[Teo]{Proposition}
\newtheorem{Cor}[Teo]{Corollary}

\theoremstyle{definition}
\newtheorem{Def}[Teo]{Definition}

\newtheorem{Q}[Teo]{Question}
\newtheorem{Rem}[Teo]{Remark}

\newtheorem{Diss}[Teo]{Discussion}
\newtheorem{Notation}[Teo]{Notation}
\begin{document}
\title{Associated primes of local cohomology of flat extensions with regular fibers
and \SigmaFinite{} $D$-modules}
\author{Luis N\'u\~nez-Betancourt}
\maketitle
\abstract{
In this article, we study the following question raised by Mel Hochster:
let $(R,m,K)$ be a local ring  and $S$ be a flat extension with regular closed fiber.
Is $\cV(mS)\cap\Ass_S H^i_I(S)$
finite for every ideal $I\subset S$ and $i\in \NN?$
We prove that the 
answer is positive when $S$ is either a polynomial or a power series ring over $R$ and $\dim(R/I\cap R)\leq 1.$
In addition, we analyze when this question can be reduced to the case where $S$ is a power series ring over $R$.
An important tool for our proof is the use of \SigmaFinite{} $D$-modules, which are not necessarily finitely generated as $D$-modules, 
but whose associated primes are finite. We give examples of this class of $D$-modules and applications to local cohomology.
}
\section{Introduction}
Throughout this article $A,R$ and $S$ will always denote commutative Noetherian rings with unit. 
If $M$ is an $S$-module and $I\subset S$ is an ideal, we denote the $i$-th local
cohomology of $M$ with support in $I$ by $H^i_I (M)$.
The structure of these modules has been widely studied 
by several authors \cite{HL,LyuFMod,LyuUMC,Marley,NunezDS,O,P-S}. 
Among the results obtained is that the set of associated primes of $H^i_I(R)$ is finite for certain regular rings.
Huneke and Sharp proved this for characteristic $p>0$ \cite{H-S}. Lyubeznik showed  this finiteness property
for regular local rings of equal characteristic zero  and finitely generated regular algebras over a field of characteristic zero
\cite{LyuDMod}. We point out that this property does not necessarily hold for ring that are not regular \cite{Moty,SiSw}.
Motivated by these finiteness results,  Mel Hochster raised the following related questions:
\begin{Q}\label{Q1}
Let $(R,m,K)$ be a local ring  and $S$ be a flat extension with regular closed fiber.
Is $$
\Ass_S H^0_{mS}(H^i_I(S))=\cV(mS)\cap H^i_I(S)
$$
finite for every ideal $I\subset S$ and $i\in \NN?$
\end{Q}
\begin{Q}\label{Q2}
Let $(R,m,K)$ be a local ring  and $S$ denote either $R[x_1,\ldots, x_n]$ or $R[[x_1,\ldots, x_n]]$. Is
$$
\Ass_S H^0_{mS}(H^i_I(S))=\cV(mS)\cap H^i_I(S)
$$
finite for every ideal $I\subset S$ and $i\in \NN?$
\end{Q}

It is clear that Question \ref{Q2} is a particular case of Question \ref{Q1}. 
In Proposition \ref{Questions}, we show that under minor additional hypothesis these questions are equivalent.  
Question \ref{Q2}
has a positive answer when $R$ 
is a ring of dimension $0$ or $1$ of any characteristic \cite{Nunez}.
In her thesis \cite{Hannah}, Robbins answered  Question \ref{Q2} positively for
certain algebras of dimension smaller than or equal to $3$ in characteristic $0$. In addition, several of her results can be obtained in  characteristic $p>0$, by
working in the category $C(S,R)$ (see the discussion after Remark \ref{RemQuotientDMod}). 
In addition, a positive answer for  Question \ref{Q1} would help to that the associated primes of local cohomology modules, $H^i_I(R),$
over certain regular local rings of mixed characteristic, $R.$ For example, $$\frac{V[[x,y,z_1,\ldots,z_n]]}{(\pi-xy)V[[x,y,z_1,\ldots,z_n]]}=\left(\frac{V[[x,y]]}{(\pi - xy)V[[x,y]]}\right)[[z_1,\ldots,z_n]],$$
where $(V,\pi V,K)$ is a complete DVR of mixed characteristic.
This is, to the best of our knowledge, the simplest example of a regular local 
ring of ramified mixed characteristic that the finiteness of $\Ass_RH^i_I(R)$ is unknown.

In this manuscript, we give a partial positive answer  for Question \ref{Q1} and \ref{Q2}. Namely:

\begin{Teo}\label{MainOne}
Let $(R,m,K)$ be any local ring. Let $S$ denote either $R[x_1,\ldots, x_n]$ or $R[[x_1,\ldots, x_n]]$. Then,
$\Ass_S H^0_{mS}H^i_I(S)$
is finite for every ideal $I\subset S$ such that $\dim R/I\cap R\leq 1$ and every $i\in \NN.$ Moreover,
if $mS\subset \sqrt{I}$, $$\Ass_S H^{j_1}_{J_1}\cdots H^{j_\ell}_{J_\ell} H^i_I(S)$$
is finite for all ideals $J_1,\ldots, J_\ell\subset S$  and  integers $j_1,\ldots,j_\ell\in\NN.$
\end{Teo}

\begin{Teo}\label{MainTwo}
Let $(R,m,K)\to (S,\eta,L)$ be a flat extension of local rings with regular closed fiber such that $R$ contains a field.
Let $I\subset S$ be an ideal such that $\dim R/I\cap R\leq 1$. 
Suppose that the morphism induced in the completions $\widehat{R}\to\widehat{S}$ maps a coefficient field of $R$ into a coefficient field of $S$.
Then,
$$
\Ass_S H^0_m H^i_I(S)
$$
is finite for every $i\in \NN.$  
\end{Teo}

In Theorem \ref{MainTwo}, 
the hypothesis that $ \widehat{\varphi}$ maps a coefficient field of $\widehat{R}$ to a coefficient field of $\widehat{S}$ is not very restrictive.
For instance, it is
satisfied when $L$ is a separable extension of $K$ \cite[Theorem $28.3$]{Mat}. In particular, this holds when
$K$ is a field of characteristic $0$ or a perfect field of characteristic $p>0.$

A key part of the proof of Theorem \ref{MainOne} is the use of \SigmaFinite{ } $D$-modules, which are directed unions of finite 
length $D$-modules that satisfy certain condition (see Definition \ref{SigmaFinite}). One of the main properties that a \SigmaFinite{ }
$D$-module satisfies is that its set of associated primes is finite. In addition, the local cohomology of a \SigmaFinite{ } $D$-module 
is again \SigmaFinite{.} Proving that the local cohomology modules supported on $H^i_{mS}H^j_I(S)$ would answer Question \ref{Q1}.

This manuscript is organized as follows. In section \ref{Pre}, we recall some  definitions and properties of local cohomology and $D$-modules. 
Later. in Section \ref{SigmaSection}, we define \SigmaFinite{} $D$-modules and give their first properties.
In Section \ref{SecAss}, we prove that certain local cohomology modules are \SigmaFinite{;} as a consequence, we give a proof of Theorem \ref{MainOne}.
Later, in Section \ref{QSF}, we give several examples of \SigmaFinite{} $D$-modules. 
Finally, in Section \ref{SecQ}, we prove that under the certain hypothesis Question \ref{Q1} and \ref{Q2} are equivalent.  
In addition, we prove Theorem \ref{MainTwo}.
\section{Preliminaries}\label{Pre}
\subsection{Local cohomology}

Let $R$ be a ring, $I\subset R$ an ideal, and  $M$ an $R$-module.
If $I$ is generated by $f_1,\ldots, f_\ell \in R$, the local cohomology
group, $H^i_I(M)$, can be computed using the $\check{\mbox{C}}$ech complex, $\Cech(\underline{f};M)$,

$$
0\to M\to \oplus_j M_{f_j}\to \ldots \to M_{f_1 \cdots f_\ell} \to 0.
$$

Let $\cK(f_1,\ldots, f_s;M)$ denote the Koszul complex associated to the sequence $\underline{f}=f_1,\ldots, f_\ell$. 
In Figure \ref{LimitKoszul} there is a direct limit involving $\cK(f^t_i;M)$, whose limit is $\Cech(f_i;M).$

\begin{figure}[h]\label{LimitKoszul}
$$
\hskip -15mm
\begin{array}{*{18}{c@{\,}}}
M & \to & M & \to & M & \to & M & \to & M & \to & \ldots \\
\downarrow  & & f_i \downarrow  & & f^2_i\downarrow & & f^3_i\downarrow  & & f^4_i \downarrow \\
M & \FDer{f_i} & M & \FDer{f_i} & M & \FDer{f_i} & M & \FDer{f_i} & M & \FDer{f_i} &\ldots \\
\end{array}
$$ 
\caption{Direct limit of Koszul complexes}
\end{figure}

Let $\underline{f}^t$ denote the sequence $f^t_1,\ldots,f^t_s$.
Since 
$$
\cK(\underline{f};M)=\cK(f_1;M)\otimes_R \ldots \otimes_R \cK(f_\ell;M),
$$
 we have that
\begin{align*}
\Cech(\underline{f};M)&=\Cech(f_1;M)\otimes_S \ldots \otimes_S \Cech(f_s;M)\\
&=\lim\limits_{\to t} \cK(f^t_1;M)\otimes_S \ldots \otimes_S \lim\limits_\to \cK(f^t_s;M)\\
&=\lim\limits_{\to t} \cK(f^t_1;M)\otimes_S \ldots \otimes_S \cK(f^t_s;M).\\
\end{align*}
We define the \emph{cohomological dimension of $I$} by
$$
\cd_R{I}=\Max\{i\mid H^i_(R)\neq 0\}.
$$
\subsection{D-modules}

Given two commutative rings $R$ and $S$ such that $R\subset S$,  we define the \emph{ring of $R$-linear differential operators of $S$,} $D(S,R)$, as 
the subring of $\Hom_R(S,S)$ obtained inductively as follows.  
The differential operators of order zero are morphisms induced by multiplication by elements in $S$ ($\Hom_S(S,S)=S$).   
An element $\theta \in \Hom_R(S,S)$ is a differential operator of 
order less than or equal to $k+1$ if $\theta\cdot r -r\cdot\theta$
is a differential operator of order less than or equal to $k$ for every $r\in S$. 

We recall that if $M$ is a $D(S,R)$-module, then $M_f$ has the structure of a $D(S,R)$-module  such that, for every $f \in S$, the natural
morphism $M\to M_f$ is a morphism of $D(S,R)$-modules. 
As a consequence,  $H^{i_1}_{I_1}\cdots H^{i_\ell}_{I_\ell}(S)$
is also a $D(S,R)$-module \cite[Examples $2.1$]{LyuDMod}.

\begin{Rem}\label{RemQuotientDMod}
If $(R,m,K)$ is a local ring and $S$ is either $R[x_1,\ldots,x_n]$ or $R[[x_1,\ldots,x_n]]$, 
then $$D(S,R)=R\left[\frac{1}{t!} \frac{\partial^t }{\partial x_i ^t} \ | \ t \in \NN, 1 \leq i \leq n \right]\subset \Hom_R(S,S)$$
\cite[Theorem $16.12.1$]{EGA}. 
Then, there is a natural surjection $\rho :D(S,R)\to D(S/IS,R/IR)$ for every ideal $I\subset R$. Moreover,
\begin{itemize}
\item[(i)] If $M$ is a $D(S,R)$-module, then $IM$ is a $D(S,R)$-submodule and the structure of $M/IM$ as a $D(S,R)$-module
is given by $\rho$, i.e., $\delta \cdot v=\rho(\delta) \cdot v$ for all $\delta\in D(S,R)$ and $v\in M/IM$.\\
\item[(ii)] If $ R$ contains the rational numbers $D(S,R)$ is a Notherian ring. Let $\Gamma_i=\{\delta\in D(S,R)\mid \Ord(\delta)\leq i\}.$ We have that
$\gr^{\Gamma} D=S[y_1,\ldots,y_n],$ which is Notherian and then so $D$ is.
\end{itemize}
\end{Rem}

We recall a subcategory of $D(S,R)$-modules introduced by Lyubeznik \cite{LyuFreeChar}. We denote by $C(S,R)$
the smallest subcategory of $D(S,R)$-modules that
contains $S_f$ for all $f\in S$ and that is closed under subobjects, 
extensions and quotients. In particular, the kernel, image and cokernel of a morphism of $D(S,R)$-modules
that belongs to $C(S,R)$ are also objects in $C(S,R)$. We note that if $M$ is an object in $C(S,R)$,
then $H^{i_1}_{I_1}\cdots H^{i_\ell}_{I_\ell}(M)$ is also an object in this subcategory; in particular, 
$H^{i_1}_{I_1}\cdots H^{i_\ell}_{I_\ell}(S)$ belongs to $C(S,R)$ 
 \cite[Lemma $5$]{LyuFreeChar}.\\

A $D(S,R)$-module, $M$, is \emph{simple} if its only $D(S,R)$-submodules are $0$ and $M$.
We say that a $D(S,R)$-module, $M$, has \emph{finite length} if there is a strictly ascending chain  of $D(S,R)$-modules, 
$0\subset M_0 \subset M_1\subset \ldots \subset M_h =M,$
called \emph{a composition series},
such that $M_{i+1}/M_i$ is a nonzero simple $D(S,R)$-module for every $i=0,\ldots, h$. In this case, $h$ is independent of the filtration and it is called the 
\emph{length} of $M$. Moreover,  the \emph{composition factors}, $M_{i+1}/M_i,$  are the same, up to permutation and isomorphism, for every filtration.\\

\begin{Notation}
If $M$ is a $D(S,R)$-module of finite length, we denote the set of its composition factors by
$\Comp{M}$.
\end{Notation}

\begin{Rem}\label{RemFinLen}
\begin{itemize}
\item[(i)] If $M$ is a nonzero simple $D(S,R)$-module, then $M$ has only one associated prime. This is because $H^0_P(M)$ is a $D(S,R)$-submodule of $M$ for every prime ideal 
$P\subset S$. As a consequence, if $M$ is a $D(S,R)$-module of finite length, then
$\Ass_S M\subset \bigcup_{N\in\Comp{M}} \Ass_S N,$ which is finite. 
\item[(ii)] If $0\to M'\to M\to M''\to 0$  is a short exact sequence of $D(S,R)$-modules of finite length, then
$\Comp{M}=\Comp{M'}\bigcup \Comp{M''}$.
\end{itemize}
\end{Rem}

\section{\SigmaFinite{} $D$-modules}\label{SigmaSection}
\begin{Notation}
Thorough this section $(R,m,K)$ denotes a local ring and $S$ denotes either $R[x_1,\ldots,x_n]$ or $R[[x_1,\ldots,x_n]]$. In addition, $D$
denotes $D(S,R).$
\end{Notation}

\begin{Def}\label{SigmaFinite}
Let $M$ be a $D$-module supported at $mS$ and $\cM$ be the set of all $D$-submodules of $M$ that have finite length.
We say that  $M$ is \emph{\SigmaFinite}  if: 
\begin{itemize}
\item[(i)] $\bigcup_{N\in\cM} N=M,$
\item[(ii)] $\bigcup_{N\in\cM}\Comp{N}$ is finite, and
\item[(iii)] For every $N\in \Comp{M}$ and $L\in \Comp{N}$, $L\in C(S/mS,R/mR)$.
\end{itemize}
We denote the set of \emph{composition factors of $M$}, $\bigcup_{N\in\cM}\Comp{N}$, by $\Comp{M}$.
\end{Def}

\begin{Rem}\label{SF-AssPrimes}
We have that
$$\Ass_S M\subset \bigcup_{N\in\Comp{M}} \Ass_S M$$
for every   \SigmaFinite{}  
$D$-module, $M$. In particular, $\Ass_S M$ is finite.
\end{Rem}
\begin{Lemma}\label{FinLenFinGen}
Let 
$M$ be a \SigmaFinite{} $D$-module and $N$ be a $D$-submodule of $S$. 
Then, $N$ has finite length as $D$-module if and only if $N$ is a finitely generated as $D$-module.
\end{Lemma}
\begin{proof} 
Suppose that $N$ is finitely generated. Let $v_1,\ldots,v_\ell$ be a set the generators of $N.$
Since $\bigcup_{N\in\cM} N=M$, there exists a finite length module $N_i$ that contains $v_i.$
Then, $N\subset N_1+\ldots +N_\ell$ and it has finite length.
It is clear that if $N$ has finite length then it is finitely generated.
\end{proof}
\begin{Prop}\label{SF-SES}
Let $0\to M'\to M\to M''\to 0$ be a short exact sequence of $D$-modules. If $M$ is \SigmaFinite{,}
then $M'$ and $M''$ are \SigmaFinite{.}
 Moreover, $\Comp{M}=\Comp{M'}\cup \Comp{M'}.$
\end{Prop}
\begin{proof}
We first assume that $M$ is \SigmaFinite{.} 
We have that $$M'=\bigcup_{N\in\cM} N\cap M'=\bigcup_{N'\in\cM} N'$$ and then $M$ is \SigmaFinite{} by Remark \ref{RemFinLen}.
Let $\rho $ denote the morphism $M\to M''$ and $N''\in\cM''$ and $\ell=\Length_{D} N''.$ There are $v_1,\dots,v_\ell\in N''$ such that
$N''=D\cdot v_1+\ldots +D\cdot v_\ell.$ Let $w_j$ be a preimage of $v_j$ and $N$ be the
$D$-module generated by  $w_1,\ldots,w_\ell.$ We have that $ N\to N''$ is a surjection, and that $N$ has finite length by Lemma \ref{FinLenFinGen}. 
Therefore, $M''=\bigcup_{N\in \cM} \rho(N)=\bigcup_{N''\in\cM''} N''$ and the result follows by Remark \ref{RemFinLen}.
\end{proof}

\begin{Prop}\label{SF-SESCharZero}
Let $0\to M'\to M\to M''\to 0$ be a short exact sequence of $D$-modules. 
Suppose that $R$ contains the rational numbers.
Then, $M$ is \SigmaFinite{} if and only if 
$M'$ and $M''$ are \SigmaFinite{.}
Moreover, $\Comp{M}=\Comp{M'}\cup \Comp{M'}.$
\end{Prop}
\begin{proof}
We first assume that $M'$ and $M''$ are \SigmaFinite{.} 
Let $v\in M.$ We have a short exact sequence
$$
0\to M'\cap D\cdot v\to D\cdot v\to D\cdot\bar{v}\to 0.
$$
$M'\cap D\cdot v$ is finitely generated because $D$ is Notherian by Remark \ref{RemQuotientDMod}. Then,
$M'\cap D\cdot v$ has finite length by Lemma \ref{FinLenFinGen}, and so
$D\cdot v$ has finite length. Therefore, $M=\bigcup_{N\in\cM} N.$

Let $N\in\cM.$ Then, $N\cap M'\in\cM'$ and $\rho(N)\in \cM''.$ We have a short exact sequence
$$0\to N\cap M'\to N\to \rho(N)\to 0$$
of finite length $D$-modules,
and then  result follows by Remark \ref{RemFinLen}.  

The other direction follows from Proposition \ref{SF-SES}
\end{proof}
\begin{Prop}\label{SF-Loc}
Let  $M$ be a \SigmaFinite{}  
$D$-module. Then, $M_f$ is \SigmaFinite{}
for every $f\in S$.
\end{Prop}
\begin{proof}
Let $N\subset M_f$ be a module of finite length. We have that $N$ is a finitely generated $D$-module. 
Then there exists a finitely generated 
$D$-submodule $N'$ of $M$ such that $N\subset N'_f.$
We have that $N'_f$ has finite length and $\Comp{N'_f}=\bigcup_{V\in\Comp{N}} \Comp{V_f}$
because $V_f$ is in $C(S/mS,R/m)$ \cite{LyuFreeChar}.
Then,
$$
M_f=\bigcup_{N\subset \cM_f} N \subset \bigcup_{N\subset \cM} N_f=M_f
$$
and the result follows.
\end{proof}

\begin{Lemma}\label{SF-DirectSum}
Let  $M$  and $M'$ be  \SigmaFinite{}  
$D$-modules. Then, $M\oplus M'$ is also \SigmaFinite{.}
\end{Lemma}
\begin{proof}
It is clear that $M\oplus M'$ is supported on $mS.$
For every $(v,v')\in M\oplus M'$, there exist $N$ and $N',$ $D$-modules
of finite length. such that $v\in N$ and $v'\in N'.$ Then, $N\oplus N'\subset M\oplus M'$ has finite length and
$(v,v')\in N\oplus N'.$ Therefore, 
$$
\bigcup_{N\subset \cM, N'\subset \cM'} N\oplus N'=M\oplus M',
$$
and the $M\oplus M'$ is union of its $D$-modules of finite length.
The rest follows from Remark \ref{RemFinLen}.
\end{proof}

\begin{Cor}\label{SF-Loc Coh}
Let  $M$ be a \SigmaFinite{}  
$D$-module. Then, $H^i_I(M)$ is \SigmaFinite{}
for every ideal $I\subset S$ and $i\in\NN$.
\end{Cor} 
\begin{proof}
Let $f_1,\ldots,f_\ell$ be generators for $I$.
We have that $\Cech(\underline{f};M)$ is \SigmaFinite{} by Lemma \ref{SF-DirectSum}.
Then $H^i_I(M)$ is also \SigmaFinite{} by Proposition \ref{SF-SES}. 
\end{proof}

\begin{Prop}\label{SF-DirectSystem}
Let  $M_t$ be an inductive direct system of \SigmaFinite{}  
$D$-modules. If $\bigcup_t\Comp{M_t}$ is finite, then
$\lim\limits_{\to t} M_t$ is \SigmaFinite{} and
$\Comp{M}\subset\bigcup_t\Comp{M_t}.$ 
\end{Prop}
\begin{proof}
Let $M=\lim\limits_{\to t} M_t$ and $\varphi_t: M_t\to M$ the morphism induced by the limit.
We have that $\phi_t(M_t)$ is a \SigmaFinite{}  $D$-module by Proposition \ref{SF-SES}. 
We may replace $M_t$ by $\phi_t(M_t)$ by Remark \ref{RemFinLen}, and assume that $M=\bigcup M_t$ and $M_{t}\subset M_{t+1}.$
If $N\subset M$ has finite length as $D$-module, then it is finitely generated and there exists a $t$ 
such that $N\subset M_t.$ Therefore,
$M=\bigcup_t M_t=\bigcup_t\bigcup_{N\in\cM_t} N$
and the result follows.
\end{proof}

\section{Associated Primes}\label{SecAss}

\begin{Notation}
Throughout this section $(R,m,K)$ denotes a local ring and $S$ denotes either $R[x_1,\ldots,x_n]$ or $R[[x_1,\ldots,x_n]]$. In addition, $D$
denotes $D(S,R).$
\end{Notation}

\begin{Lemma}\label{Comp}
Let $J\subset S$ be an ideal and $M$ be an $R$-module of finite length.
Then, $H^i_{J}(M\otimes_R S )$ is a $D(S,R)$-module of finite length. Moreover,
$\Comp{H^i_{JS}(M\otimes_R S )}\subset \bigcup_{j}\Comp{H^j_{JS}(S/mS )}$.  
\end{Lemma}
\begin{proof}
Our proof will be by induction on $h=\Length_R (M)$. If $h=1$,  we have that
$H^i_{JS}(R/m\otimes_R S)=H^i_{JS}( S/mS ),$ which has finite length as a 
$D(S,R)$-module  \cite[Theorem $2,$ Corollary $3$]{LyuFreeChar}  and 
by Remark \ref{RemQuotientDMod}. Clearly,   
$$\Comp{H^i_{JS}(M\otimes_K A )}=\Comp{H^i_{JS}(R/m\otimes_K A)}=\bigcup_{j}\Comp{H^j_{JS}(S/mS )}$$ in this case.
Suppose that the statement is true for $h$ and $\Length_R (M)=h+1$.
We have a short exact sequence of $R$-modules, 
$0\to K\to M\to M'\to 0$, where $h=\Length_R (M')$. Since $S$ is flat over $R$, we have that
$0\to K\otimes_R S\to M\otimes_R S\to M'\otimes_R S\to 0$ is also exact.
Then, we have a long exact sequence
$$
\ldots \to H^i_{J}( K\otimes_R S)\to H^i_{J}(M\otimes_R S)\to H^i_{J}( M'\otimes_R S)\to \ldots $$
Then  $H^i_{J}(M\otimes_R S)$ has finite length by the induction hypothesis and Remark \ref{RemFinLen}.
In addition,
\begin{align*}
\Comp{H^i_{J} (M\otimes_R S)} & \subset \Comp{H^i_{J} (M'\otimes_R S)} \bigcup \Comp{H^i_{J} (K\otimes_R S)} \\
& \subset \bigcup_j \Comp{ H^j_{J} (S/mS)}.
\end{align*}
and the result follows by the induction hypothesis and Remark \ref{RemFinLen}.
\end{proof}

\begin{Prop}\label{SF I mS}
Let $I\subset S$ be an ideal containing $mS$. Then
$ H^i_I(S)$ is \SigmaFinite{} for every $i\in\NN.$
\end{Prop}
\begin{proof}
Let $f_1,\ldots, f_d$ be a system of parameters for $R$ 
and $g_1,\ldots, g_\ell$ be a set of generators for $I$.
Let $\underline{f}^t $ denote the sequence $f^t_1,\ldots, f^t_\ell.$
Let $T_i=\{T^{p,q}_t\}$ be the double complex of $D(S,R)$-modules given by the tensor product $\cK(\underline{f};R) \otimes_R \Cech(\underline{g};S)$. 
The direct limit  $\cK(\underline{f}^t;R)$ introduced in Figure \ref{LimitKoszul}, 
induces a direct limit of double complexes 
$\Tot(T_t)\to\Tot(T_{t+1})$. Since $\lim\limits_{\to t}\cK(\underline{f}^t;R)=\Cech(\underline{f};R)$,
we have that $\lim\limits_{\to t} \Tot(T_t)=\Cech(\underline{f,g}; S)$.
Let $E^{p,q}_{r,t}$ be the spectral sequence associated to $T_t$.
We have that
$$
E^{p,q}_{2,t}=H^{p}_{I} ( H^q (\cK(\underline{f}^t; S))  
\Rightarrow  E^{p,q}_{\infty, t} =H^{p+q}\Tot (T_t).
$$
We note that $H^q(\cK(\underline{f}^t;  S))=H^q (\cK(\underline{f}^t; R))\otimes_R S,$ because $S$ is $R$-flat.
Since $H^q (\cK(\underline{f}^t; R))$ has finite length as an $R$-module, we have that
$E^{p,q}_{2,t}$ is a $D(S,R)$-module of finite length for all $p,q\in\NN$ and that $\Comp{E^{p,q}_{2,t}}=\bigcup_j \Comp{H^i_{JS}(S/mS )}$ by Lemma \ref{Comp}. Moreover,
$E^{p,q}_{r,t}$ is a $D(S,R)$-module of finite length, and
$$\Comp{E^{p,q}_{r,t}}\subset \bigcup_{p,q} \Comp{ E^{p,q}_{2,t}}= \bigcup_{j} \Comp{H^j_{I}(S/mS )}$$ for $r>2.$
Then, $\Comp{H^{i} (\Tot (T_t))}\subset  \bigcup_{j} \Comp{H^i_{I}(S/mS)}$ for every $j,t\in \NN$
by Remark \ref{RemFinLen}; in particular, $\bigcup_t \Comp{H^i \Tot (T_t)}$ is finite and every element there belongs to $C(S/mS,R/mR)$.
Therefore, $E^{p,q}_{r,t}$ is a  \SigmaFinite{} $D(S,R)$-module.
Moreover,
$$
H^i_I(S)=H^i(\Cech(\underline{f,g}; S))=H^i(\lim\limits_{\to t} \Tot(T_t))=\lim\limits_{\to t} H^i(\Tot(T_t))
$$
because the direct limit is exact. Hence, $H^i_I(S)$ is \SigmaFinite{} by Proposition \ref{SF-DirectSystem}.
\end{proof}

\begin{Cor}\label{AssPrimes}
Let $I\subset S$ be an ideal containing $mS$ and $J_1,\ldots, J_\ell\subset S$ be any ideals. Then
$H^{j_1}_{J_1}\cdots H^{j_\ell}_{J_\ell}H^i_I(S)$ 
is \SigmaFinite{.}
\end{Cor}
\begin{proof}
This is a consequence of Proposition \ref{SF I mS} and Corollary \ref{SF-Loc Coh}.
\end{proof}
\begin{proof}[Proof of Theorem \ref{MainOne}]
This is a consequence of Remark \ref{SF-AssPrimes} and Corollary  \ref{AssPrimes}.
\end{proof}

\begin{Prop}\label{AssDimOne}
Let $(R,m,K)$ be any local ring. Let $S$ denote either $R[x_1,\ldots, x_n]$ or $R[[x_1,\ldots, x_n]]$. 
Let $I\subset S$ be an ideal, such that $\dim R/I\cap R\leq 1$. 
Then,
$$
\Ass_S H^0_m H^i_I(S)
$$
is finite for every $i\in \NN.$  
\end{Prop}
\begin{proof}
Since $\dim R/(I\cap R)\leq 1$, there exists $f\in R$ such that $mS\subset \sqrt{I+fS}.$
We have the exact sequence
$$
\ldots \to H^i_{(I,f)S}(S) \FDer{\alpha_i} H^i_I(S) \FDer{\beta_i} H^i_I(S_f)\to\ldots.
$$
Then, 
$$\Ass_S H^i_I(S)\cap \cV(mS)\subset (\Ass_S \Im(\alpha_i)\cap \cV(mS) )\bigcup (\Ass_S \Im(\beta_i)\cap \cV(mS))$$
Since $H^i_{(I,f)S}(S)$ is a \SigmaFinite{} $D(S,R)$-module by Proposition \ref{SF I mS}, we have that $\Im(\alpha_1)$ is also \SigmaFinite{} by Proposition 
\ref{SF-SES}, and so
$\Ass_S \Im(\alpha_i)$ is finite. Since $\Im(\beta_i)\subset H^i_I(S_f),$ $\Ass_S \Im(\beta_i)\cap \cV(mS)=\emptyset$.
Therefore, 
$$\Ass_S H^i_I(S)\cap \cV(mS)=\Ass_S H^0_{mS}H^i_I(S)$$
is finite.
\end{proof}

\begin{Prop}
Suppose that $R$ is a ring of characteristic $0$ and that $\dim R/(I\cap R)\leq 1.$
Then $H^j_{mS}H^i_I(S)$ is \SigmaFinite{} for every $i,j\in\NN.$ 
\end{Prop}
\begin{proof}
Since $\dim R/(I\cap R)\leq 1,$ there exists $g\in R,$ such that $mS\subset \sqrt{(I,g)S}.$ 
We have the long exact sequence 
$$
\ldots \to H^i_{(I,g)S}(S) \to H^i_I(S)\to H^i_I(S_g)\to \ldots.
$$
Let $M_i= \Ker(H^i_{(I,g)S}(S) \to H^i_I(S)),$
$N_i=\Im(H^i_{(I,g)S}(S) \to H^i_I(S))$ and
$W_i=\Im(H^i_I(S)\to H^i_I(S_g)).$
We have the following short exact sequences:
$$
0\to M_i\to H^i_{(I,g)S}(S) \to N_i\to 0,
$$
$$
0\to N_i\to H^i_{I}(S) \to W_i\to 0
$$
and
$$
0\to W_i\to H^i_{I}(S_g) \to M_{i+1}\to 0.
$$
Since $mS\subset \sqrt{(I,g)S},$ $H^i_{(I,g)S}(S)$ is \SigmaFinite{} by Proposition \ref{SF I mS}. Then,
$M_i$ and $N_i$ is \SigmaFinite{} for every $i\in\NN$ by Proposition \ref{SF-SES}.
By the long exact sequences
$$
\ldots\to H^j_{mS} (M_i)\to H^j_{mS}H^i_{(I,g)S}(S) \to H^j_{mS} (N_i)\to \ldots,
$$
$$
\ldots\to H^j_{mS} (N_i)\to H^j_{mS} H^i_{I}(S) \to H^j_{mS} (W_i))\to \ldots
$$
and
$$
\ldots\to H^j_{mS} (W_i)\to H^j_{mS}H^i_{I}(S_g) \to H^j_{mS} (M_{i+1})\to \ldots,
$$
$H^j_{mS} (M_i),$ $H^j_{mS}H^i_{(I,g)S}(S)$ and $H^j_{mS} (M_i)$ are \SigmaFinite{} for every $i,j\in\NN.$
Since $H^j_{mS}H^i_{I}(S_g)=0,$ $H^j_{mS} (W_i)=H^j_{mS} (M_{i+1}).$
Then, $H^j_{mS} (W_i)$ is \SigmaFinite{,} and so $H^j_{mS} H^i_{I}(S) $ is \SigmaFinite{} by \ref{SF-SESCharZero}.
\end{proof}


\section{More examples of \SigmaFinite{} $D$-modules}\label{QSF}
In the previous section we gave a positive answer for specific cases for Question \ref{Q2}. Our method consisted in proving that
$H^j_{mS}H^i_I(S)$ is \SigmaFinite{} and then applying Remark \ref{SF-AssPrimes}. This motivates the following question:

\begin{Q}\label{QSigma}
Is $H^j_{mS}H^i_I(S)$ \SigmaFinite{} for every ideal $I\subset S$ and $i,j\in\NN?$
\end{Q}
In this section, we provide positive examples for Question \ref{QSigma}. 
\begin{Prop}\label{SF-CD}
Let $(R,m,K)$ be any local ring. Let $S$ denote either $R[x_1,\ldots, x_n]$ or $R[[x_1,\ldots, x_n]]$. 
Let $I\subset S$ be an ideal such that $\Depth_S I =\cd_S{I}.$
Then, $H^i_{mS} H^{\Depth_S I}_I(S)$ is  \SigmaFinite{} for every $i\in \NN.$
\end{Prop}
\begin{proof}
We have that the spectral sequence 
$$
E^{p,q}_{2}=H^{p}_{mS} H^q_I (S) \Longrightarrow E^{p,q}_{\infty} =H^{p+q}_{(I,m)S}(S),
$$
converges at the second spot, because  $\Depth_S I=\cd_S{I}.$
Hence,
$$
H^{p}_{mS} H^q_I (S) =H^{p+q}_{(I,m)S}(S)
$$
and the result follows by Proposition \ref{SF I mS}.
\end{proof}

\begin{Prop}\label{SF-Ext}
Let $(R,m,K)$ be any local ring. Let $S$ denote either $R[x_1,\ldots, x_n]$ or $R[[x_1,\ldots, x_n]]$. 
Let $I\subset S$ be an ideal such that $\Ext^i_S(S/mS, H^j_I(S))$ is a $D$-module in $C(R,S)$ for every 
$i\in\NN.$
Then, $H^i_{mS} H^j_I(S)$ is a \SigmaFinite{}  $D(S,R)$-module for every $i,j\in \NN.$
\end{Prop}
\begin{proof}
We claim that $\Ext^i_S(N\otimes_R S, H^j_I(S))$ is a $D(S/mS,K)$-module in $C(S/mS,K)$ for every 
$i\in\NN$ and every finite length $R$-module $N$. Moreover, $\Comp{\Ext^i_S(N\otimes_R S, H^j_I(S))}\subset \bigcup_i \Comp{\Ext^i_S(K\otimes_R S, H^j_I(S))}.$
The proof of our claim is analogous to Lemma \ref{Comp}.

The direct system $\Ext^i(S/m^\ell S,H^j_I(S))\to \Ext^i(S/m^{\ell+1} S,H^j_I(S))$ satisfies the hypotheses of Proposition \ref{SF-DirectSystem}. Hence,
$$H^i_{mS}H^j_I(S)=\lim\limits_{\to \ell} \Ext^i(S/m^\ell S,H^j_I(S))$$
is a \SigmaFinite{} $D(S,R)$-module.
\end{proof}

\begin{Rem}
The condition that
$\Ext^i_S(S/mS, H^j_I(S))$ be a $D(S/mS,K)$-module in $C(S/mS,K)$ for every 
$i\in\NN$ is not necessary. 

Let $R=K[[s,t,u,w]]/(us+vt)$, where $K$ is a field. This is the ring given by Hartshorne's example \cite{Ha}. 
Let $I=(s,t)A$. Hartshorne showed that
$\dim_K \Hom_A(K,H^2_I(A)) $ is not finite.
Let $S$ be either $R[x_1,\ldots,x_n]$ or $R[[x_1,\ldots,x_n]]$. 
Therefore,
$$\Ext^0_S(S/mS,H^2_I(S))=\Hom_S(S/mS,H^2_I(S))$$
$$=\Hom_R(K,H^2_I(R))\otimes_R S=\oplus S/mS,$$
where the direct sum is infinite. Then, $\Ext^0_S(S/mS,H^2_I(S))$
does not belong to $C(S,R).$ 

On the other hand, $H^0_mH^2_I(S)$ is a direct limit of finite direct sums of $S/mS.$ This direct limit satisfies the hypotheses of 
Proposition \ref{SF-DirectSystem}. Therefore, $H^0_mH^2_I(S)$ is a \SigmaFinite{} $D(S,R)$-module.
\end{Rem}

\begin{Prop}
Let $(R,m,K)$ be any local ring and let $S$ denote $R[x_1,\ldots, x_n].$ 
Let $I\subset S$ be an ideal.
Then, $H^i_{mS} H^0_I(S)$ is  \SigmaFinite{} for every $i\in \NN.$
In addition,
if $\cd_S{I}\leq 1,$ then
$H^i_{mS} H^j_I(S)$ is  \SigmaFinite{} for every $i,j\in \NN.$
\end{Prop}
\begin{proof}
We claim that there exists an ideal $J\subset R$ such that
$H^0_I(S)=JS.$
We have that $H^0_I(S)$ is a $D(S,R)$-module. For every $f=\sum_{\alpha} c_\alpha x^\alpha\in H^0_I(S)$ and $\partial \in D(S,R),$
$\partial f\in H^0_I(S).$ Therefore, $c_{\alpha}\in H^0_I(S).$
and $H^0_I(S)=JS,$ where $J=\{c_{\alpha} \mid \sum_{\alpha} c_\alpha x^\alpha\in  H^0_I(S)\}.$

We have that 
\begin{align*}
\Ext^i_S(S/mS, H^0_I(S))&=\Ext^i_S(R/mR\otimes_R S, J\otimes_R S)\\
&=\Ext^i_R(K, J)\otimes_S S\\
&=\oplus^\mu S/mS, \hbox{ where } \mu=\dim_{K} \Ext^i_R(K,J),
\end{align*}
and it is a $D(S,R)$-module in $C(S,R)$ for every 
$i\in\NN.$ The first claim follows from Proposition \ref{SF-Ext}.

We have that $H^1_I(S)=H^1_{I(S/J)}(S/J)$  \cite[Corollary $2.1.7$]{BroSharp}.
In addition, $S/JS=(R/J)[x_1,\ldots,x_n]$ and 
$$\Depth_{I(S/JS)}=\cd_{I(S/JS)}(S/JS)=1.$$ The second claim follows from Proposition \ref{SF-CD}.
\end{proof}

\section{Reduction to power series rings}\label{SecQ}

\begin{Diss}\label{Discussion}
Suppose that $(R,m,K)$ and  $(S,\eta,L)$ are complete local rings  and that  $\varphi:R\to S$ is a flat extension of local rings  with regular closed fiber.
Assume that $\widehat{\varphi}$  maps a coefficient field of $R$ to a coefficient field of $S.$ 
We pick such coefficient fields, and then $\varphi(K)\subset L$.
Thus, $R=K[[x_1,\ldots,x_n]]/I$ for some ideal $I\subset K[[x_1,\ldots,x_n]]$.  
Let  $A=L\ctp_K R=L[[x_1,\ldots,x_n]]/I L[[x_1,\ldots,x_n]].$  
We note that $A$ is a  flat local extension of $R,$ such that $mA$ is the maximal ideal of $A$.
Let $\theta:A\to S$ be the morphism induced by $\varphi$ and our choice of coefficient fields.

We claim that $S$ is a flat $A$-algebra.
Let $F_*$ be a free resolution of $R/mR.$ Then, $A\otimes_R F_*$ is a free resolution for $A/mA.$ We have that
$$
\Tor^{A}_1(S,A/mA)=H_1(S\otimes_A A\otimes_R F_*)
$$
$$=H_1(S\otimes_R F_*)=\Tor^{R}_1(S,R/mR)=0
$$ 
because $S$ is a flat extension.
Since $mA$ is the maximal ideal of $A$, we have that $S$ is a flat $A$-algebra by  
the local criterion of flatness \cite[Theorem $6.8$]{Eisenbud}.

Let $d=\dim(S/mS)$ and $z_1,\ldots,z_d\in S$ be  preimages of a regular system of parameters for $S/mS$. Let $\varphi:A[[y_1,\ldots,y_d]]\to S$ 
be the morphism given by sending $A$ to 
$S$ via $\theta$ and  $y_i$ to  $z_i.$ Since $$(mA+(z_1,\ldots,z_d)A)S=\eta$$ and the morphism induced by $\varphi$ in the  quotient fields of 
$A$ and $S$ is an isomorphism. Hence, $\varphi$ is an isomorphism. 
\end{Diss}
\begin{Prop}\label{Questions}
Questions \ref{Q1} and \ref{Q2} are equivalent when we restrict them to a local extensions, such that the induced morphism in the 
completions maps a coefficient field of the domain to a coefficient field of the target. 
\end{Prop}
\begin{proof}
Let $\varphi:(R,m,K)\to (S,\eta,L)$ be a flat extension of local rings with regular closed fiber.
Suppose that  $\widehat{\varphi}:\widehat{R}\to \widehat{S},$ the induced morphism in the 
completions, maps a coefficient field of the $\widehat{R}$ to a coefficient field of $\widehat{S}$. 
We have that
$\Ass_R H^0_{mR}H^i_I(S )$ is finite if and only if 
$\Ass_R H^0_{m\widehat{R}}H^i_I(\widehat{S})$ is finite.
Let $A$ be as in the previous discussion and $d=\dim(S/mS).$
The result follows, because $\widehat{S}=A[[y_1,\ldots,y_d]]$ and $mS=(mA)S.$
\end{proof}

\begin{Rem}
In the previous proposition, 
the hypothesis that $ \widehat{\varphi}$ maps a coefficient field of  $\widehat{R}$ to a coefficient field of $\widehat{S}$ is satisfied when
$L$ is a separable extension of $K$   \cite[Theorem $28.3$]{Mat}.  
\end{Rem}

\begin{proof}[Proof of Theorem \ref{MainTwo}]
By Discussion \ref{Discussion}, we may assume that $R$ is complete and $S$ is a power series ring over $R$.
The rest is a consequence of Proposition \ref{AssDimOne}.
\end{proof}

\section*{Acknowledgments}
I would like to thank Mel Hochster for sharing  Questions \ref{Q1} and \ref{Q2} with me and for his invaluable comments and suggestions.  
I also wish to thank Emily Witt and Wenliang Zhang for useful mathematical conversations related to this work.
Thanks are also due to the National Council of Science and Technology of Mexico by its support through grant $210916.$

{\sc Department of Mathematics, University of Michigan, Ann Arbor, MI $48109$--$1043$, USA.}\\
{\it Email address:}  \texttt{luisnub@umich.edu}

\begin{thebibliography}{ZZZZZ}
\bibitem{BroSharp}
M.P. Brodmann and R. Sharp, Local cohomology: an algebraic introduction with geometric applications. Cambridge Studies in Advanced Mathematics, $60.$ 
Cambridge University Press, Cambridge, $1998.$ 
\bibitem{Eisenbud}
D. Eisenbud, Commutative algebra with a view toward algebraic geometry. Graduate Texts in Mathematics, $150.$ Springer-Verlag, New York, $1995.$
\bibitem{EGA}
A. Grothendieck,  \'El\'ements de g\'eom\'etrie alg\'ebrique IV: \'Etude locale des sch\'emas et des morphismes de sch\'emas, Quatri\'me partie.  Publications Math\'ematiques I.H.E.S. $32$ ($1967$). EGA 4 IV, PIHES 32 (1967). 
\bibitem{Ha}
R. Hartshorne, Affine Duality and Cofiniteness, Invent. Math. $9$, $145$--$164$ ($1970$).
\bibitem{HL}
C. Huneke and G. Lyubeznik. On the vanishing of local cohomology modules. Invent.
Math., $102$($1$):$73$--$93,$ $1990$.
\bibitem{H-S}
C. Huneke and R. Sharp, Bass numbers of local cohomology modules.
Transactions of the American Mathematical Society, Vol. 339, No. 2, October 1993), pp.
765-779.
\bibitem{Moty}
M. Katzman, An example of an infinite set of associated primes of a local cohomology module. 
J. Algebra $252$ ($2002$), no. $1,$ $161$–-$166.$ 
\bibitem{LyuDMod}
G. Lyubeznik, Finiteness properties of local cohomology modules (an application of $D$-modules to commutative algebra). Invent. Math. 
$113$ ($1993$), no. $1,$ $41$--$55.$ 
\bibitem{LyuFMod}
G. Lyubeznik, $F$-modules: applications to local cohomology and $D$-modules 
in characteristic $p>0$. J. Reine Angew. Math. $491$ ($1997$), $65$--$130$.
\bibitem{LyuFreeChar}
G. Lyubeznik, Finiteness properties of local cohomology modules:
a characteristic-free approach. J. Pure Appl. Algebra $151$ ($2000$), no. $1,$ $43$--$50.$
\bibitem{LyuUMC}
G. Lyubeznik, Finiteness properties of local cohomology modules
for regular local rings of mixed characteristic: the unramified case.
Special issue in honor of Robin Hartshorne. Comm. Algebra $28$ ($2000$), no. $12$, $5867$--$5882$.
\bibitem{Marley}
T. Marley,
The associated primes of local cohomology modules over rings of small dimension.
Manuscripta Math. $104$ ($2001$), no. $4,$ $519$–-$525.$ 
\bibitem{NunezDS}
L. N\'u\~nez-Betancourt, Local cohomology properties of direct summands, J. Pure Appl. Algebra $216$ ($2012$) no. $10,$ $2137$--$2140.$
\bibitem{Mat}
H. Matsumura, Commutative ring theory. Cambridge Studies in Advanced Mathematics, 8. Cambridge University Press, Cambridge, 1986.
\bibitem{Nunez}
L. N\'u\~nez-Betancourt, Local cohomology modules of polynomial or power series rings over a ring of small dimension.
Preprint ($2012$).
\bibitem{O}
A. Ogus, Local cohomological dimension of algebraic varieties. Ann. Math., $98,$ $327$--$365$
($1973$).
\bibitem{P-S}
C. Peskine, L. Szpiro, Dimension projective finite et cohomologie locale. I.H.E.S., $42,$
$323$--$395$ ($1973$).
\bibitem{Hannah}
H. Robbins, (2008) Finiteness of associated primes of local cohomology modules, Ph.D. Thesis. University of Michigan.
\bibitem{SiSw}
A. Singh, I. Swanson, Associated primes of local cohomology modules and of Frobenius powers, Int. Math. Res. Not. $2004$, 
no. $33$, $1703$--$1733$.
\end{thebibliography}
\end{document}